\pdfoutput=1
\RequirePackage{ifpdf}
\ifpdf 
\documentclass[pdftex]{sigma}
\else
\documentclass{sigma}
\fi

\numberwithin{equation}{section}

\newtheorem{Theorem}{Theorem}[section]

\newtheorem{Lemma}[Theorem]{Lemma}
\newtheorem{Proposition}[Theorem]{Proposition}
{ \theoremstyle{definition}
\newtheorem{Definition}[Theorem]{Definition}

\newtheorem{Example}[Theorem]{Example}
\newtheorem{Remark}[Theorem]{Remark} }

\usepackage[all]{xy}

\makeatletter
\newcommand{\subscripts}[3]{%
 \@mathmeasure\z@\displaystyle{#2}%
 \global\setbox\@ne\vbox to\ht\z@{}\dp\@ne\dp\z@
 \setbox\tw@\box\@ne
 \@mathmeasure4\displaystyle{\copy\tw@_{#1}}%
 \@mathmeasure6\displaystyle{{#2}_{#3}}%
 \dimen@-\wd6 \advance\dimen@\wd4 \advance\dimen@\wd\z@
 \hbox to\dimen@{}\mathop{\kern-\dimen@\box4\box6}%
}
\makeatother

\newcommand{\prarrow}[2]{\ar@<0.5ex>[r]^-{#1} \ar@<-0.5ex>[r]_-{#2}}
\newcommand{\plarrow}[2]{\ar@<0.5ex>[l]^-{#1} \ar@<-0.5ex>[l]_-{#2}}
\newcommand{\pdarrow}[2]{\ar@<0.5ex>[d]^-{#1} \ar@<-0.5ex>[d]_-{#2}}
\newcommand{\puarrow}[2]{\ar@<0.5ex>[u]^-{#1} \ar@<-0.5ex>[u]_-{#2}}

\begin{document}

\allowdisplaybreaks

\newcommand{\arXivNumber}{2410.05846}

\renewcommand{\PaperNumber}{070}

\FirstPageHeading

\ShortArticleName{Mikami--Weinstein Type Theorem for Cosymplectic Groupoid Actions}

\ArticleName{Mikami--Weinstein Type Theorem for Cosymplectic\\ Groupoid Actions}

\Author{Shuhei YONEHARA}

\AuthorNameForHeading{S.~Yonehara}

\Address{National Institute of Technology, Yonago College, Tottori, 683-8502, Japan}
\Email{\href{shuhei.yonehara.0201@gmail.com}{shuhei.yonehara.0201@gmail.com}}
\URLaddress{\url{https://sites.google.com/view/shuheiyonehara}}

\ArticleDates{Received January 20, 2025, in final form August 06, 2025; Published online August 16, 2025}

\Abstract{The Mikami--Weinstein theorem is a generalization of the classical Marsden--Weinstein--Meyer symplectic reduction theorem to the case of symplectic groupoid actions. In this paper, we introduce the notion of a cosymplectic groupoid action on a cosymplectic manifold and prove a theorem which is a natural analogue of the Mikami--Weinstein theorem.}

\Keywords{cosymplectic manifolds; cosymplectic groupoids; momentum maps; Hamiltonian actions}

\Classification{53D17; 53D20; 22A22; 58H05}

\section{Introduction}
Since the pioneering work of Marsden--Weinstein and Meyer \cite{marsden1974reduction,meyer1973symmetries}, many types of reduction theorems have been studied for various geometric structures on manifolds. Albert \cite{albert1989theoreme} studied Hamiltonian actions on cosymplectic manifolds, which are odd-dimensional analogues of symplectic manifolds (see \cite{cappelletti2013survey} for more details about cosymplectic manifolds) and proved a reduction theorem. On the other hand, Mikami--Weinstein \cite{mikami1988moments} generalized the Marsden--Weinstein--Meyer theorem to symplectic groupoid actions, which extends the notion of a Hamiltonian action on symplectic manifolds.

In this paper, we define a notion of an action of a \textit{cosymplectic groupoid} on a cosymplectic manifold by using the notion of a~\textit{Lagrangian--Legendrian submanifold}. Afterwards, we prove a reduction theorem which is an analogue of the Mikami--Weinstein theorem. The notion of a~cosymplectic groupoid is introduced by \cite{djiba2015cosymplectic} and recently studied in \cite{fernandes2023cosymplectic}. They are defined as Lie groupoids whose space of arrows is endowed with a multiplicative cosymplectic structure.

\begin{table}[htbp]\renewcommand{\arraystretch}{1.2}
 \centering
 \begin{tabular}{|c||c|c|}
 \hline
 &Phase space&Symmetry\\ \hline
 Marsden--Weinstein--Meyer &symplectic manifold& Lie group\\
 Mikami--Weinstein &symplectic manifold & symplectic groupoid \\
 Albert & cosymplectic manifold & Lie group\\
 this paper&cosymplectic manifold &cosymplectic groupoid \\ \hline
 \end{tabular}
\end{table}

This paper is organized as follows. In Section~\ref{cosymplectic}, we briefly recall cosymplectic structures and the reduction theorem by Albert. In Section~\ref{groupoid}, we recall the notion of a symplectic groupoid and its role in Poisson geometry. In addition, we review the definition and some properties of cosymplectic groupoids. In Section~\ref{action}, we introduce the notion of a Lagrangian--Legendrian submanifold of cosymplectic manifolds, and define cosymplectic actions of cosymplectic groupoids on cosymplectic manifolds. We observe that if a cosymplectic groupoid $\mathcal{G}=(G_1\rightrightarrows G_0)$ acts on a cosymplectic manifold $M$, then a symplectic groupoid $S_\mathcal{G}=(S_{G_1}\rightrightarrows G_0)$, where $S_{G_1}$ is the symplectic leaf of $G_1$ that contains unit arrows, acts on each symplectic leaf of $M$. In Section~\ref{Theorem}, we prove the main theorem.

\begin{Theorem}[Theorem~\ref{thm:3}]
 Let $\mathcal{G}=(G_1\rightrightarrows G_0)$ be a cosymplectic groupoid and $(M,\eta,\omega)$ a cosymplectic, free and proper left $\mathcal{G}$-module with respect to a momentum map $\rho\colon M\to G_0$. Assume that $\xi\in\rho(M)$ is a regular value of $\rho$. Then $(S_{\mathcal{G}})_\xi\backslash \rho^{-1}(\xi)$ admits a unique cosymplectic structure $\bigl(\eta^\xi,\omega^\xi\bigr)$ such that $p^\ast\eta^\xi=\eta|_{\rho^{-1}(\xi)}$, $ p^\ast\omega^\xi=\omega|_{\rho^{-1}(\xi)}$, where $(S_{\mathcal{G}})_\xi$ is the isotropy group at $\xi$ and $p\colon\rho^{-1}(\xi)\to (S_{\mathcal{G}})_\xi\backslash \rho^{-1}(\xi)$ is the quotient map.
\end{Theorem}

In Section~\ref{examples}, we give examples of our main theorem. The main example reconstructs Albert's cosymplectic reduction theorem. Lastly, in Section~\ref{morita}, we mention Morita equivalence of cosymplectic groupoids and show potential for future research.

\section{Cosymplectic manifolds}\label{cosymplectic}
An \textit{almost cosymplectic structure} on a $(2n+1)$-dimensional manifold $M$ is a pair of $\eta\in\Omega^1(M)$ and $\omega\in\Omega^2(M)$ such that $\eta\wedge\omega^n\neq0$. On an almost cosymplectic manifold $(M,\eta,\omega)$, there is a~unique vector field $R$ which satisfies
$\omega(R,-)=0$, $\eta(R)=1$.
The vector field $R$ is called the \textit{Reeb vector field} of $(M,\eta,\omega)$. Moreover, we have an isomorphism of $C^\infty(M)$-modules $\flat\colon\mathfrak{X}(M)\to \Omega^1(M)$ defined by $\flat(X)=\omega(X,-)+\eta(X)\eta$. Conversely, a pair $(\eta,\omega)$ is an almost cosymplectic structure if and only if the map $\flat\colon\mathfrak{X}(M)\to \Omega^1(M)$ defined as above is an isomorphism and there is a vector field $R$ which satisfies the above conditions.

An almost cosymplectic structure $(\eta,\omega)$ is called a contact structure when $\omega={\rm d}\eta$.
On the~other hand, an almost cosymplectic structure $(\eta,\omega)$ is called a \textit{cosymplectic structure} when ${{\rm d}\eta=0}$, $ {\rm d}\omega=0$.

For a contact structure $\eta\in\Omega^1(M)$, the distribution $\operatorname{Ker}\eta$ is completely non-integrable. On the other hand, for a cosymplectic structure $(\eta,\omega)$, the distribution $\operatorname{Ker}\eta$ is integrable since $\eta$ is closed. Therefore, contact structures and cosymplectic structures are two classes of almost cosymplectic structures which are polar opposites of each other.

Two cosymplectic manifolds $(M_1,\eta_1,\omega_1)$ and $(M_2,\eta_2,\omega_2)$ are said to be \textit{isomorphic} if there is a diffeomorphism $f\colon M_1\to M_2$ which satisfies $f^\ast\eta_2=\eta_1$ and $f^\ast\omega_2=\omega_1$. Then $f$ is called an \textit{isomorphism} of cosymplectic manifolds.

Let $(M_1,\eta_1,\omega_1)$ and $(M_2,\eta_2,\omega_2)$ be two cosymplectic manifolds and $\dim M_1=2n+1$, $\dim M_2=2m+1$. Then a pair $(\eta,\omega)$ of forms defined by
$\eta=\eta_1+\eta_2$, $\omega=\omega_1+\omega_2+\eta_1\wedge {\rm d}t$
is a~cosymplectic structure on $M_1\times M_2\times\mathbb{R}$, where $t$ denotes the coordinate of $\mathbb{R}$. In fact, $\eta$ and~$\omega$ are closed and
\[\eta\wedge\omega^{n+m+1}=(n+m+1)\omega_1^n\wedge\omega_2^m\wedge\eta_2\wedge\eta_1\wedge {\rm d}t\neq0\]
holds.

A cosymplectic structure $(\eta,\omega)$ on $M$ induces a Poisson structure $\pi\in\Gamma\bigl(\wedge^2 TM\bigr)$ on $M$ by~$\pi(\alpha,\beta)=\omega\bigl(\flat^{-1}\beta,\flat^{-1}\alpha\bigr)$, where $\alpha,\beta\in T^\ast M$. This Poisson structure is regular and has corank 1. Its symplectic leaves coincide with those of the integrable distribution $\operatorname{Ker}\eta$ and the symplectic form on a symplectic leaf $S$ is $\omega|_S$. In fact, it is known that a cosymplectic structure on $M$ is equivalent to a corank 1 regular Poisson structure on $M$ with a Poisson vector field which is transverse to the symplectic leaves \cite{guillemin2011codimension}.

For every function $f\in C^\infty(M)$ on a cosymplectic manifold $M$, we can associate a vector field~$X_f$ by
$X_f=\flat^{-1}({\rm d}f-R(f)\eta)$.
$X_f$ is called the \textit{Hamiltonian vector field} of $f$. This condition is equivalent to
$\omega(X_f,-)={\rm d}f-R(f)\eta$, $\eta(X_f)=0$. $X_f$ coincides with the usual notion of the Hamiltonian vector field determined by the Poisson bivector; namely, $X_f=\pi^\sharp({\rm d}f)$ holds, where the map $\pi^\sharp\colon T^\ast M\to TM$ is defined by $\beta\bigl(\pi^\sharp(\alpha)\bigr)=\pi(\alpha,\beta)$.

Let $(M,\eta,\omega)$ be a cosymplectic manifold, $G$ a Lie group acts on $M$ from left, and $L_g\colon M\to M$ the map of left action by $g\in G$. We suppose that the action preserves $\eta$, $\omega$, i.e., $L_g^\ast \eta=\eta$, $ L_g^\ast \omega=\omega$. Denote the Lie algebra of $G$ as $\mathfrak{g}$. Albert \cite{albert1989theoreme} defined the notion of a momentum map on cosymplectic manifolds.

\begin{Definition}
 A smooth map $\mu\colon M\to\mathfrak{g}^\ast$ is called a \textit{momentum map} when the following conditions are satisfied:
 \begin{itemize}\itemsep=0pt
 \item $\mu$ is equivariant, i.e., $\mu(gx)=\textrm{Ad}^\ast_g\mu(x)$ holds for any $x\in M$ and $g\in G$.
 \item For any $A\in\mathfrak{g}$, the induced vector field $A^\ast\in\mathfrak{X}(M)$ is the Hamiltonian vector field of a~function $\mu^A\colon M\to\mathbb{R}$ defined by $\mu^A(x)=(\mu(x))(A)$,
 \item For the Reeb vector field $R$ and any $A\in\mathfrak{g}$, ${\rm d}\mu^A(R)=0$ holds.
 \end{itemize}
\end{Definition}

The action of $G$ is said to be \textit{Hamiltonian} if there is a momentum map. Now we assume that there is a Hamiltonian action of $G$ on $(M,\eta,\omega)$ which is free and proper. Let $\xi\in\mathfrak{g}^\ast$ be a regular value of a momentum map $\mu\colon M\to\mathfrak{g}^\ast$. Since $\mu$ is equivariant, the isotropy group~$G_\xi$ acts on~$\mu^{-1}(\xi)$. We denote the quotient $G_\xi\backslash\mu^{-1}(\xi)$ as $M^\xi$ and the natural projection as~${p\colon\mu^{-1}(\xi)\to M^\xi}$. The following theorem is an analogue of Marsden--Weinstein--Meyer theorem.

\begin{Theorem}[Albert \cite{albert1989theoreme}]\label{thm:0}
 There is a unique cosymplectic structure $\bigl(\eta^\xi,\omega^\xi\bigr)$ on $M^\xi$ which satisfies $p^\ast\eta^\xi=\eta|_{\mu^{-1}(\xi)}$, $ p^\ast\omega^\xi=\omega|_{\mu^{-1}(\xi)}$.
\end{Theorem}

\begin{Remark}
 In \cite{yonehara2024reduction}, reduction theorems of coK\"{a}hler manifolds and 3-cosymplectic manifolds are proved. They are natural odd-dimensional analogues of the reduction theorems of K\"{a}hler manifolds and hyperK\"{a}hler manifolds \cite{hitchin1987hyperkahler}, respectively.
\end{Remark}

\section{Symplectic groupoids and cosymplectic groupoids}\label{groupoid}

A groupoid is a small category in which all arrows are invertible. This is summarized in the following diagram
\[\xymatrix{G_1\subscripts{s}{\times}{t}G_1\ar[r]_-m&G_1 \ar@(dl,dr)^i \prarrow{s}{t} &G_0 \ar[r]_-u & G_1,}\]
where $G_1\subscripts{s}{\times}{t}G_1=\{(g,h)\in G_1\times G_1\mid s(g)=t(h)\}$.\footnote{Throughout the paper, we will use this ``fibered product'' notation without explanation.} $G_1$ is a~set of arrows and $G_0$ is a~set of objects, $m$, $i$, $s$, $t$, $u$ (these maps are called structure maps of the groupoid) are maps of multiplication, inverse, source, target, and unit, respectively. $G_1$ and $G_0$ are sometimes called the total space and the base space, respectively. For any $\xi\in G_0$, $s^{-1}(\xi)\cap t^{-1}(\xi)$ is a~group. This group is called the \textit{isotropy group} on $\xi$, and denoted by $\mathcal{G}_\xi$. We simply denote a~groupoid~${\mathcal{G}=(G_1,G_0,m,i,s,t,u)}$ by $\mathcal{G}=(G_1\rightrightarrows G_0)$, $m(g,h)$ by $gh$, $u(\xi)$ by $1_\xi$ for $\xi\in G_0$.

A groupoid is called a \textit{Lie groupoid} if $G_1$ and $G_0$ are smooth manifolds, $s$, $t$ are smooth submersions, and $m$, $i$, $u$ are smooth maps. A Lie groupoid $H_1\rightrightarrows H_0$ is called a \textit{Lie subgroupoid} of another Lie groupoid $G_1\rightrightarrows G_0$ when $H_1\rightrightarrows H_0$ is a subcategory of $G_1\rightrightarrows G_0$ and $H_1$ is an immersed submanifold of $G_1$. A \textit{morphism} between Lie groupoids is a smooth functor.

In Poisson geometry, there is an important class of Lie groupoids, namely, symplectic group\-oids. Roughly speaking, symplectic groupoids are ``integration'' of Poisson manifolds.
\begin{Definition}
 A \textit{symplectic groupoid} is a pair $(G_1\rightrightarrows G_0, \omega_{G_1})$ of a Lie groupoid and a~symplectic form on $G_1$ which is \textit{multiplicative}, i.e.,
 $m^\ast\omega_{G_1}=\textrm{pr}_1^\ast\omega_{G_1}+\textrm{pr}_2^\ast\omega_{G_1}$
 holds, where $\textrm{pr}_i\colon G_1\subscripts{s}{\times}{t}G_1\allowbreak\to G_1$ denotes the natural projections.
\end{Definition}

The space of objects $G_0$ of a symplectic groupoid $(G_1\rightrightarrows G_0, \omega_{G_1})$ has a unique \textit{integrable}\footnote{A Poisson manifold is said to be integrable when induced Lie algebroid (cotangent bundle) is integrable by a~Lie groupoid. For more details, see~\cite{crainic2011lectures}, for example.} Poisson structure such that the source map is a Poisson map. Conversely, Mackenzie and Xu~\cite{mackenzie2000integration} proved that for any integrable Poisson manifold $G_0$, there exists a unique (up to isomorphism) symplectic groupoid $(G_1\rightrightarrows G_0, \omega_{G_1})$ whose $s$-fiber $s^{-1}(\xi)$ on each $\xi\in G_0$ is simply connected (such a Lie groupoid is said to be $s$-simply connected), and these operations are inverses of each other. So there is a correspondence
\[\{s\text{-simply connected symplectic groupoid}\}\overset{\text{$1\!:\!1$}}\longleftrightarrow\{\text{integrable Poisson manifold}\}.\]

\begin{Example}
 Let $G_0$ be a manifold and $G$ a Lie group acting on $G_0$ from left. Then one obtains a Lie groupoid $G\times G_0\rightrightarrows G_0$ by defining the following structure maps:
 \begin{gather*}
 s(g,\xi)=\xi,\qquad t(g,\xi)=g\xi,\qquad 1_{\xi}=(e,\xi),\\
 (g,h\xi)(h,\xi)=(gh,\xi),\qquad i(g,\xi)=\bigl(g^{-1},g\xi\bigr),
 \end{gather*}
 where $g,h\in G$, $ \xi\in G_0$ and $e$ is the unit of $G$. The Lie groupoid $G\times G_0\rightrightarrows G_0$ is called the \textit{action groupoid} associated to the Lie group action.

 Let $\mathfrak{g}$ be the Lie algebra of $G$. There is a left $G$-action $\rm{Ad^\ast\colon G\to\rm{GL}(\mathfrak{g}^\ast)}$ on $\mathfrak{g}^\ast$ called the \textit{coadjoint action}. Consider the action groupoid associated to this action. The space of arrows~${G\times\mathfrak{g}^\ast\simeq T^\ast G}$ has the canonical symplectic form, and $G\times\mathfrak{g}^\ast\rightrightarrows\mathfrak{g}^\ast$ is a symplectic groupoid by this symplectic form. In this case, the corresponding Poisson structure on the space of objects $\mathfrak{g}^\ast$ is the \textit{linear Poisson structure}, which is defined by
 $\{f,g\}(\xi)=\xi([{\rm d}f_\xi,{\rm d}g_\xi])$
 for $f,g\in C^\infty(\mathfrak{g}^\ast)$ and $\xi\in\mathfrak{g}^\ast$, where $[\cdot,\cdot]$ is the Lie bracket of $\mathfrak{g}$ and we consider ${\rm d}f_\xi,{\rm d}g_\xi\colon \mathfrak{g}^\ast\to\mathbb{R}$ identifying $T_\xi\mathfrak{g}^\ast$ with $\mathfrak{g}^\ast$.
\end{Example}

The notion of a cosymplectic groupoid is defined in exactly the same way as that of a symplectic groupoid:
\begin{Definition}
 A \textit{cosymplectic groupoid} is a triplet $(G_1\rightrightarrows G_0, \eta_{G_1}, \omega_{G_1})$ of a Lie groupoid and a cosymplectic structure on $G_1$ such that
 \[m^\ast\eta_{G_1}=\textrm{pr}_1^\ast\eta_{G_1}+\textrm{pr}_2^\ast\eta_{G_1}\qquad m^\ast\omega_{G_1}=\textrm{pr}_1^\ast\omega_{G_1}+\textrm{pr}_2^\ast\omega_{G_1}\]
 holds.
\end{Definition}

\begin{Example}
 Let $G_1\rightrightarrows G_0$ be a Lie groupoid and $G$ an abelian Lie group. Then a pair $({P\rightrightarrows G_0},\allowbreak (p,\textrm{id}_{G_0}))$ of a Lie groupoid $P\rightrightarrows G_0$ and a morphism $(p,\textrm{id}_{G_0})\colon (P\rightrightarrows G_0)\!\to\!{(G_1\rightrightarrows G_0)}$ is called a \textit{central extension} of $G_1\rightrightarrows G_0$ by $G$ when $G$ acts on $P$ and the map $p\colon P\to G_1$ is a~principal $G$-bundle.

 For any symplectic groupoid $(G_1\rightrightarrows G_0, \omega_{G_1})$, let us consider a central extension $({P\rightrightarrows G_0},\allowbreak (p,\textrm{id}_{G_0}))$ by $G=\mathbb{R}$ or $G=\mathbb{S}^1$. Let $\eta_P$ be a multiplicative, flat connection form of the principal bundle $p\colon P\to G_1$. Then $(P\rightrightarrows G_0,\eta_P,\omega_P)$ is a cosymplectic groupoid, where~${\omega_P=p^\ast\omega_{G_1}}$. In~particular, the trivial $\mathbb{R}$-central extension $(G_1\times\mathbb{R}\rightrightarrows G_0,\textrm{pr}_\mathbb{R}^\ast {\rm d}t,\textrm{pr}_{G_1}^\ast\omega_{G_1})$, where $\textrm{pr}$ denotes the projections, is a cosymplectic groupoid.
\end{Example}

The space of arrows of a cosymplectic groupoid has a symplectic foliation defined by the distribution $\operatorname{Ker}\eta$ and there is a distinguished symplectic leaf:
\begin{Theorem}[\cite{fernandes2023cosymplectic}]\label{thm:1}
 Let $\mathcal{G}=(G_1\rightrightarrows G_0)$ be a cosymplectic groupoid. Then any unit arrow in~$G_1$ belongs to the same symplectic leaf $S_{G_1}$. Moreover, $S_\mathcal{G}:=(S_{G_1}\rightrightarrows G_0)$ is a Lie subgroupoid of $\mathcal{G}$ and it is a symplectic groupoid.
\end{Theorem}

\section{Actions of cosymplectic groupoids}\label{action}

The notion of an action of a Lie groupoid on a manifold $M$ is a generalization of the situation where an action of Lie group $G$ on $M$ and an equivariant map $\rho\colon M\to G_0$, where $G_0$ is another manifold on which $G$ acts, is given.
\begin{Definition}
 Let $\mathcal{G}=(G_1\rightrightarrows G_0)$ be a Lie groupoid and $M$ be a manifold. A \textit{left action} of~$\mathcal{G}$ on $M$ is a pair $(\rho,\Phi)$ of smooth maps $\rho\colon M\to G_0$ and $\Phi\colon G_1\subscripts{s}{\times}{\rho}M\to M$ which satisfies the following conditions:
 \begin{enumerate}\itemsep=0pt
 \item[(1)] $\rho(\Phi(g,x))=t(g)$ when $(g,x)\in G_1\subscripts{s}{\times}{\rho}M$,

 \item[(2)] $\Phi(g,\Phi(h,x))=\Phi(gh,x)$ when $(g,h)\in G_1\subscripts{s}{\times}{t}G_1$, $ (h,x)\in G_1\subscripts{s}{\times}{\rho}M$,

 \item[(3)] $\Phi(1_{\rho(x)},x)=x$ for any $x\in M$.
 \end{enumerate}

 Hereinafter, we simply denote $\Phi(g,x)$ by $gx$ and refer to $M$ as \textit{left $\mathcal{G}$-module}. The map $\rho\colon M\to G_0$ is called a \textit{momentum map}. A right action of $\mathcal{G}$ on $M$ is also defined similarly, by swapping the role of the source map and the target map.
\end{Definition}

A left $\mathcal{G}$-action on $M$ (or a left $\mathcal{G}$-module $M$) is said to be
\begin{itemize}\itemsep=0pt
 \item \textit{free} if $gx=x$ (for some $x\in M$ such that $(g,x)\in G_1\subscripts{s}{\times}{\rho}M$) implies $g=1_{\rho(x)}$,

 \item \textit{proper} if a map $G_1\subscripts{s}{\times}{\rho}M\to M\times M;\ (g,x)\mapsto(gx,x)$ is proper.
\end{itemize}The orbit space $\mathcal{G}\backslash M$ of a free and proper Lie groupoid action is a smooth manifold and the quotient map $M\to\mathcal{G}\backslash M$ is a submersion. Then one also calls $M\to\mathcal{G}\backslash M$ a \textit{principal $\mathcal{G}$-bundle} (see~\cite{moerdijk2003introduction}). In particular, for any regular value $\xi\in G_0$ of $\rho$, the isotropy Lie group~${\mathcal{G}_\xi=s^{-1}(\xi)\cap t^{-1}(\xi)}$ smoothly acts on $\rho^{-1}(\xi)$, and the quotient map $\rho^{-1}(\xi)\to\mathcal{G}_\xi\backslash\rho^{-1}(\xi)$ is a~submersion to the smooth quotient space.

Let $(\mathcal{G},\omega_{G_1})$ be a symplectic groupoid and $(M,\omega)$ a symplectic manifold. A left $\mathcal{G}$-action on~$M$ (or a left $\mathcal{G}$-module $M$) is said to be \textit{symplectic} if the graph of the action, i.e.,
\[\{(g,x,gx)\in G_1\times M\times M\mid (g,x)\in G_1\subscripts{s}{\times}{\rho}M\}\]
is a Lagrangian submanifold of $(G_1\times M\times M,\omega_{G_1}+\omega_1-\omega_2)$, where $\omega_i$ denotes the symplectic structure of $i$-th $M$.

\begin{Remark}
 The condition that the graph is a Lagrangian submanifold is grounded in Weinstein's ``symplectic creed'' \cite{weinstein1981symplectic} philosophy.
\end{Remark}

\begin{Example}\label{ex:1}
 Let $G_0$ be a manifold and $G$ a Lie group acting on $G_0$ from left. Consider the action groupoid $\mathcal{G}=(G\times G_0\rightrightarrows G_0)$. Then we obtain a correspondence
 \[\{\text{left}\ \mathcal{G}\text{-action on}\ M\}\overset{\text{$1\!:\!1$}}
 \longleftrightarrow\{\text{left}\ G\text{-action on}\ M\ \text{with a equivariant map} \ \rho\colon M\to G_0\}\]
 by a formula $gx=(g,\rho(x))x$, where the left side means the action of $g\in G$ on $x\in M$ and the right side means the action of $(g,\rho(x))\in G\times G_0$ on $x$. Moreover, when $\mathcal{G}$ is $G\times\mathfrak{g}^\ast\rightrightarrows\mathfrak{g}^\ast$ and~$M$ has a symplectic form $\omega$, we have a correspondence
 \[\{\text{symplectic left}\ \mathcal{G}\text{-action on}\ (M,\omega)\}\overset{\text{$1\!:\!1$}}\longleftrightarrow\{\text{Hamiltonian left}\ G\text{-action on}\ (M,\omega)\}\]
 (see \cite{crainic2021lectures}, for example).
\end{Example}

In order to define the notion of a cosymplectic groupoid action, we need to consider an analogue of Lagrangian submanifolds.

Let $(M,\eta,\omega)$ be a cosymplectic manifold and $N\subset M$ a submanifold. Then we call $N$ a~\textit{Lagrangian--Legendrian submanifold} or in short, \textit{LL submanifold} if
$T_p N\subset \operatorname{Ker}\eta_p$ (Legendrian property),
\smash{$(T_p N)^{\omega_p|_{\operatorname{Ker}\eta_p}}=T_p N$} (Lagrangian property)
holds for any $p\in N$, where \smash{$(T_p N)^{\omega_p|_{\operatorname{Ker}\eta_p}}$} denotes the orthogonal complement of $T_p N$ with respect to $\omega_p|_{\operatorname{Ker}\eta_p}$.

In fact, the notion of a LL submanifold is defined for \textit{almost} cosymplectic manifolds. In the case of contact manifolds, the definition of a LL submanifold coincides with that of a Legendrian submanifold.

\begin{Remark}
 An embedding $\iota\colon N\hookrightarrow M$ to a cosymplectic manifold $(M,\eta,\omega)$ is a LL sub\-mani\-fold if $\iota^\ast\eta=0$, $\iota^\ast\omega=0$ and $\dim M=2\dim N+1$ holds.
\end{Remark}

\begin{Lemma}
 Let $(M_1,\eta_1,\omega_1)$ and $(M_2,\eta_2,\omega_2)$ be two cosymplectic manifolds and $f\colon M_1\to M_2$ a diffeomorphism. Then $f$ is an isomorphism of cosymplectic manifolds if and only if the graph of $f$, i.e.,
 \[\Gamma:=\{(x,f(x),1)\in M_1\times M_2\times\mathbb{R}\mid x\in M_1\}\]
 is a LL submanifold of a cosymplectic manifold $(M_1\times M_2\times\mathbb{R},\eta,\omega)$, where
 \[\eta=\eta_1-\eta_2,\qquad\omega=\omega_1-\omega_2+\eta_1\wedge {\rm d}t.\]
\end{Lemma}
\begin{proof}
 The graph $\Gamma$ is the image of an embedding
 $\iota\colon M_1\to M_1\times M_2\times\mathbb{R}$, $\iota(x)=(x,f(x),1)$. Then we obtain
 \begin{gather*}
 \iota^\ast{\eta}=\iota^\ast(p_1^\ast\eta-p_2^\ast\eta)=\eta-f^\ast\eta,\qquad
 \iota^\ast{\omega}=\iota^\ast(p_1^\ast\omega-p_2^\ast\omega+(p_1^\ast\eta)\wedge {\rm d}q)=\omega-f^\ast\omega,
 \end{gather*}
 where $p_i$ and $q$ denotes projections to $M_i$ and $\mathbb{R}$, respectively. In addition, we have
 \[2\dim\Gamma+1=2\dim M_1+1=\dim(M_1\times M_2\times\mathbb{R}).\]
 Therefore, $\Gamma$ is a LL submanifold if and only if $f^\ast\eta_2=\eta_1,\ f^\ast\omega_2=\omega_1$.
\end{proof}

We can also rephrase the definition of a cosymplectic groupoid by using the notion of a LL submanifold.

\begin{Proposition}\label{prop:1}
 Let $\mathcal{G}=(G_1\rightrightarrows G_0)$ be a Lie groupoid and $(\eta,\omega)$ a cosymplectic structure on~$G_1$. Then a triplet $(\mathcal{G},\eta,\omega)$ is a cosymplectic groupoid if and only if
 the graph of the multiplication, i.e.,
 \[\Gamma:=\{(g,h,1,gh,1)\in G_1\times G_1\times\mathbb{R}\times G_1\times\mathbb{R}\mid (g,h)\in G_1\subscripts{s}{\times}{t}G_1\}\]
 is a LL submanifold of a cosymplectic manifold $(G_1\times G_1\times\mathbb{R}\times G_1\times\mathbb{R},\widetilde{\eta},\widetilde{\omega})$, where
 \begin{gather*}\widetilde{\eta}:=\eta_{1}+\eta_2-\eta_3,\qquad
 \widetilde{\omega}:=(\omega_{1}+\omega_2+\eta_{1}\wedge{\rm d}t_1)-\omega_3+(\eta_{1}+\eta_2)\wedge{\rm d}t_2
 \end{gather*}
 $(t_i$ denotes the coordinate of $i$-th $\mathbb{R}$ and $(\eta_i,\omega_i)$ denotes the cosymplectic structure of $i$-th $G_1)$.
\end{Proposition}
\begin{proof}
 The graph $\Gamma$ is the image of an embedding $\iota\colon G_1\subscripts{s}{\times}{t}G_1\to G_1\times G_1\times\mathbb{R}\times G_1\times\mathbb{R}$ given by $\iota(g,h)=(g,h,1,gh,1)$. Then we obtain
 \begin{gather*}\iota^\ast\widetilde{\eta}=\iota^\ast(p_1^\ast\eta+p_2^\ast\eta-p_3^\ast\eta)={\rm pr}_1^\ast\eta+{\rm pr}_2^\ast\eta-m^\ast\eta,\\
 \iota^\ast\widetilde{\omega}=\iota^\ast(p_1^\ast\omega+p_2^\ast\omega-p_3^\ast\omega+(p_1^\ast\eta)\wedge {\rm d}q_1+(p_1^\ast\eta+p_2^\ast\eta)\wedge{\rm d}q_2)={\rm pr}_1^\ast\omega+{\rm pr}_2^\ast\omega-m^\ast\omega,
 \end{gather*}
 where $p_i$ and $q_i$ denotes projections to $i$-th $G_1$ and $i$-th $\mathbb{R}$, respectively, and ${{\rm pr}_i\colon G_1\subscripts{s}{\times}{t}G_1\to G_1}$ also denotes projections. Hence the multiplicativity of $\eta$ and $\omega$ is equivalent to $\iota^\ast\widetilde{\eta}=0$ and~${\iota^\ast\widetilde{\omega}=0}$, respectively. In addition, we have
$\dim\Gamma=\dim(G_1\subscripts{s}{\times}{t}G_1)=2\dim G_1-\dim G_0$,
 and since $\dim G_1=2\dim G_0+1$ (see \cite{fernandes2023cosymplectic}), we obtain
 \[
 2\dim\Gamma+1=4\dim G_1-2\dim G_0+1=3\dim G_1+2=\dim(G_1\times G_1\times\mathbb{R}\times G_1\times\mathbb{R}).
\]
Therefore, the multiplicativity condition is equivalent to $\Gamma$ being a LL submanifold.
\end{proof}

Now we can define a notion of a cosymplectic groupoid action.

\begin{Definition}
 Let $(\mathcal{G}=(G_1\rightrightarrows G_0),\eta_{G_1},\omega_{G_1})$ be a cosymplectic groupoid and $(M,\eta,\omega)$ a~cosymplectic manifold. A left $\mathcal{G}$-action on $M$ (or a left $\mathcal{G}$-module $M$) is said to be \textit{cosymplectic} if the following conditions are satisfied:
 \begin{enumerate}\itemsep=0pt
 \item[(1)] The momentum map $\rho\colon M\to G_0$ of the action satisfies ${\rm d}\rho(R)=0$, where $R$ is the Reeb vector field of $(M,\eta,\omega)$.
 \item[(2)] The graph of the action, i.e.,
 \[\Gamma:=\{(g,x,1,gx,1)\in G_1\times M\times\mathbb{R}\times M\times\mathbb{R}\mid (g,x)\in G_1\subscripts{s}{\times}{\rho}M\}\]
 is a LL submanifold of a cosymplectic manifold $(G_1\times M\times\mathbb{R}\times M\times\mathbb{R},\widetilde{\eta},\widetilde{\omega})$, where
 \[\widetilde{\eta}:=\eta_{G_1}+\eta_1-\eta_2,\qquad\widetilde{\omega}:=(\omega_{G_1}+\omega_1+\eta_{G_1}\wedge {\rm d}t_1)-\omega_2+(\eta_{G_1}+\eta_1)\wedge {\rm d}t_2\]($t_i$ denotes the coordinate of $i$-th $\mathbb{R}$ and $(\eta_i,\omega_i)$ denotes the cosymplectic structure of $i$-th~$M$).
 \end{enumerate}
\end{Definition}

\begin{Remark}
 Condition (1) means that the Reeb vector field $R$ lies in the direction of the fibers of the momentum map $\rho$, and this will later be necessary for constructing a cosymplectic structure on our reduced space.
\end{Remark}

The following proposition is essentially used in Section~\ref{Theorem} for the proof of our main theorem.

\begin{Proposition}\label{prop:2}
 Let $(\mathcal{G}=(G_1\rightrightarrows G_0),\eta_{G_1},\omega_{G_1})$ be a cosymplectic groupoid, $(M,\eta,\omega)$ a cosymplectic left $\mathcal{G}$-module and $(\rho,\Phi)$ its action maps. Let $S_\mathcal{G}=(S_{G_1}\rightrightarrows G_0)$ be the symplectic subgroupoid obtained by Theorem~{\rm\ref{thm:1}}. Then any symplectic leaf $S$ of $(M,\eta,\omega)$ is a symplectic left $S_\mathcal{G}$-module by action maps
 \[\rho|_S\colon \ S\to G_0,\qquad
 \Phi|_{S_{G_1}\subscripts{s}{\times}{\rho} S}\colon
\ S_{G_1}\subscripts{s}{\times}{\rho} S\to S.\]
\end{Proposition}
\begin{proof}
 Firstly, we see that the Legendrian property of the graph $\Gamma$ of the action $(\rho,\Phi)$ implies~${\Phi(S_{G_1}\subscripts{s}{\times}{\rho} S)\subset S}$. Let $(g,x)\in S_{G_1}\subscripts{s}{\times}{\rho} S$ and $(g(t),x(t))$ be a smooth path in $S_{G_1}\subscripts{s}{\times}{\rho} S$ whose starting point is $(1_{\rho(x)},x)$ and ending point is $(g,x)$. Then we obtain a smooth path $({g}(t),{x}(t),1,{(gx)}(t),1)$ in $\Gamma$ and
 \[0=\widetilde{\eta}(\dot{g}(t),\dot{x}(t),0,(\dot{gx})(t),0)=\eta_{G_1}(\dot{g}(t))+\eta(\dot{x}(t))-\eta((\dot{gx})(t))=-\eta((\dot{gx})(t))\]holds. Therefore, two points $x=1_{\rho(x)}x$ and $gx$ are in the same symplectic leaf $S$.

 Secondly, we see that the Lagrangian property of the graph $\Gamma$ implies that the restricted action $(\rho|_S,\Phi|_{S_{G_1}\subscripts{s}{\times}{\rho} S})$ is symplectic. Let $({g}(t),{x}(t))$ be a smooth path in $S_{G_1}\subscripts{s}{\times}{\rho} S$. Then we have
 \[0=\widetilde{\omega}(\dot{g}(t),\dot{x}(t),0,(\dot{gx})(t),0)=\omega_{G_1}(\dot{g}(t))+\omega(\dot{x}(t))-\omega((\dot{gx})(t)).\]In addition to this, taking the dimension count into consideration, we can see that the graph of the $S_\mathcal{G}$-action on $S$ is a Lagrangian submanifold.
\end{proof}

\section{Mikami--Weinstein type theorem}\label{Theorem}

Now we recall the statement of the Mikami--Weinstein theorem.\footnote{In \cite{xu2}, an alternative proof utilizing Morita equivalence was given.}

\begin{Theorem}[\cite{mikami1988moments}]\label{thm:2}
 Let $\mathcal{G}=(G_1\rightrightarrows G_0)$ be a symplectic groupoid and $M$ a symplectic, free and proper left $\mathcal{G}$-module with respect to a momentum map $\rho\colon M\to G_0$. Assume that $\xi\in\rho(M)$ is a regular value of $\rho$. Then $\mathcal{G}_\xi\backslash \rho^{-1}(\xi)$ is a symplectic manifold. Moreover, if $\rho$ is submersive, the family of symplectic manifolds \smash{$\big\{\mathcal{G}_\xi\backslash \rho^{-1}(\xi)\big\}_{\xi\in\rho(M)}$} is precisely the symplectic foliation of the Poisson manifold $\mathcal{G}\backslash M$.
\end{Theorem}

\begin{Example}
 Consider the case of $\mathcal{G}=(G\times\mathfrak{g}^\ast\rightrightarrows\mathfrak{g}^\ast)$. For any $\xi\in\mathfrak{g}^\ast$, ${\mathcal{G}_\xi\simeq\{g\in G\mid{\rm Ad}^\ast_g\xi=\xi\}}$. Thus by the correspondence in Example~\ref{ex:1}, we can see that the Marsden--Weinstein--Meyer theorem is a special case of the Mikami--Weinstein theorem.
\end{Example}

The following is our main theorem.

\begin{Theorem}\label{thm:3}
 Let $\mathcal{G}=(G_1\rightrightarrows G_0)$ be a cosymplectic groupoid and $(M,\eta,\omega)$ a cosymplectic, free and proper left $\mathcal{G}$-module with respect to a momentum map $\rho\colon M\to G_0$. Assume that~${\xi\in\rho(M)}$ is a regular value of $\rho$. We denote $(S_{\mathcal{G}})_\xi\backslash \rho^{-1}(\xi)$ as $M^\xi$ and the quotient map as $p\colon\rho^{-1}(\xi)\to M^\xi$. Then $M^\xi$ admits a unique cosymplectic structure $\bigl(\eta^\xi,\omega^\xi\bigr)$ such that $p^\ast\eta^\xi=\eta|_{\rho^{-1}(\xi)}$ and $p^\ast\omega^\xi=\omega|_{\rho^{-1}(\xi)}$.
\end{Theorem}
\begin{proof}
 Let $\{S_i\}_{i\in I}$ be the symplectic foliation of $M$. Since the Reeb vector field $R$ of $M$ satisfies ${\rm d}\rho(R)=0$, each $S_i$ intersects transversely with $\rho^{-1}(\xi)$, and thus $(\rho|_{S_i})^{-1}(\xi)$ is a smooth manifold.

 By Proposition~\ref{prop:2}, the symplectic groupoid $S_\mathcal{G}$ acts on each leaf $S_i$ symplectically. Hence $\big\{S^\xi_i:=(S_{\mathcal{G}})_\xi\backslash (\rho|_{S_i})^{-1}(\xi)\big\}_{i\in I}$ forms a foliation on $M^\xi$ of codimension 1 (see \cite[Section 1.3]{moerdijk2003introduction}). In addition, we can apply Theorem~\ref{thm:2} on each leaf and thus \smash{$\big\{S^\xi_i\big\}_{i\in I}$} is a symplectic foliation on \smash{$M^\xi$}.

 Let $L_g\colon\rho^{-1}(\xi)\to\rho^{-1}(\xi)$ be the left action map by $g\in(S_{\mathcal{G}})_\xi$ and $x(t)$ a integral curve of $R$ in $\rho^{-1}(\xi)$. Then by the Legendrian property of the graph,
 \[\eta((L_g)_\ast R)=\eta((\dot{gx})(t))=\eta_{G_1}(0)+\eta(\dot{x}(t))=\eta(R)=1\]holds. Similarly, by the Lagrangian property of the graph, we have $\omega((L_g)_\ast R,-)=0$ and thus~$R$ is left invariant. Hence $R$ descends to a vector field ${R}^\xi:={\rm d}p(R)$ on the quotient $M^\xi$. The vector field ${R}^\xi$ is transverse to the symplectic foliation on $M^\xi$.

 The reduced foliation \smash{$\big\{S^\xi_i\big\}_{i\in I}$} is coorientable since \smash{$\{S_i\}_{i\in I}$} is. We choose a defining 1-form~$\eta^\xi$ of the foliation \smash{$\big\{S^\xi_i\big\}_{i\in I}$} such that $\eta^\xi\bigl(R^\xi\bigr)=1$ holds. Then we have \smash{$p^\ast\eta^\xi=\eta|_{\rho^{-1}(\xi)}$}. Let $\omega_i$ be the symplectic form on \smash{$S^\xi_i$}. Then we define a 2-form $\omega^\xi$ on $M^\xi$ by
 $\omega^\xi\bigl(R^\xi,-\bigr)=0$, \smash{$\omega^\xi|_{S^\xi_i}=\omega_i$}. Then we have \smash{$p^\ast\omega^\xi=\omega|_{\rho^{-1}(\xi)}$}. $\eta^\xi,\omega^\xi$ are closed since $\eta,\omega$ are closed and $p$ is a submersion.

 Moreover, since $\eta^\xi\bigl(R^\xi\bigr)\neq0$ and $\omega_i^n\neq0$, $\eta^\xi\wedge\bigl(\omega^\xi\bigr)^n$ is a volume form, where $n$ is an integer such that $\dim M^\xi=2n+1$. Therefore, a pair $\bigl(\eta^\xi,\omega^\xi\bigr)$ is a cosymplectic structure on $M^\xi$.
\end{proof}

\section{Examples}\label{examples}

In this section, we give two examples of Theorem~\ref{thm:3}.

\begin{Example}
 Let $\mathcal{G}=(G_1\rightrightarrows G_0)$ be a cosymplectic groupoid. Then $\mathcal{G}$ acts on $G_1$ by the multiplication of groupoid with $t\colon G_1\to G_0$ as the momentum map. This action is free, proper and cosymplectic. In fact, the graph of the action is a LL submanifold because of Proposition~\ref{prop:1}, and the Reeb vector field $R$ of $G_1$ satisfies $R\in \operatorname{Ker}{\rm d}t$ (see \cite{fernandes2023cosymplectic}).

 For any $\xi\in G_0$, the reduced cosymplectic manifold $(S_{\mathcal{G}})_\xi\backslash t^{-1}(\xi)$ is obtained by Theorem~\ref{thm:3}. Here the symplectic leaf $(S_{\mathcal{G}})_\xi\backslash (t|_{S_{G_1}})^{-1}(\xi)$ coincides with the $S_\mathcal{G}$-orbit in $G_0$ through $\xi$. We~can see that it is also a leaf of the symplectic foliation of a Poisson manifold $G_0$ by Theorem~\ref{thm:2}, and thus we have two foliated manifolds each having the orbit as a leaf.
\end{Example}

\begin{Example}
 Let $G$ be a Lie group acts on a cosymplectic manifold $(M,\eta,\omega)$ freely and properly. We assume that there is a momentum map $\mu\colon M\to\mathfrak{g}^\ast$ with respect to the action. Then let us consider a cosymplectic groupoid
 \[T^\ast G\times\mathbb{R}\simeq G\times\mathfrak{g}^\ast\times\mathbb{R}\rightrightarrows\mathfrak{g}^\ast\]
 (the trivial $\mathbb{R}$-central extension of a symplectic groupoid $T^\ast G\rightrightarrows\mathfrak{g}^\ast$).

 For any $\varepsilon>0$, we define
 \begin{gather*}
 M_\varepsilon=\{x\in M \mid \text{Reeb flow}\ \varphi_t(x)\ \text{is defined in}\ t\in[-\varepsilon,\varepsilon] \},\\
 \mathcal{G}_\varepsilon=(G\times\mathfrak{g}^\ast\times(-\varepsilon,\varepsilon)\rightrightarrows\mathfrak{g}^\ast).
 \end{gather*}
 In fact, although $\mathcal{G}_\varepsilon$ is not a Lie groupoid, it is a \textit{local Lie groupoid} (i.e., the composition of arrows is defined only in a neighborhood of the unit arrows) whose structure maps are the same as those of $G\times\mathfrak{g}^\ast\times\mathbb{R}\rightrightarrows\mathfrak{g}^\ast$, and the previously discussed concepts related to actions can also be applied to local Lie groupoids. We can define a cosymplectic $\mathcal{G}_\varepsilon$-action on $M_\varepsilon$ by
 \[(g,\xi,t)\cdot x:=\varphi_t(gx)\]
 for $(g,\xi,t)\in G\times\mathfrak{g}^\ast\times(-\varepsilon,\varepsilon)$, $ x\in M_\varepsilon$, with $\mu|_{M_\varepsilon}\colon M_\varepsilon\to\mathfrak{g}^\ast$ as the momentum map. In this case, Theorem~\ref{thm:3} coincides with Theorem~\ref{thm:0} for the $G$-action on $(M_\varepsilon,\eta|_{M_\varepsilon},\omega|_{M_\varepsilon})$.
\end{Example}

\section{Further study: Morita equivalence}\label{morita}

We defined the notion of a cosymplectic groupoid action, thus we can also define the notion of Morita equivalence between cosymplectic groupoids as in the case of symplectic groupoids \cite{xu2}.

\begin{Definition}
 A cosymplectic groupoid $\mathcal{G}=({G}_1\rightrightarrows G_0)$ is said to be \textit{Morita equivalent} to another cosymplectic groupoid $\mathcal{H}=({H}_1\rightrightarrows H_0)$ when there is a cosymplectic manifold $M$, a~left cosymplectic $\mathcal{G}$-action and a right cosymplectic $\mathcal{H}$-action on $M$ which satisfies the following conditions:
 \begin{enumerate}\itemsep=0pt
 \item[(1)] Momentum maps $\rho\colon M\to G_0$ and $\sigma\colon M\to H_0$ are surjective submersions.

 \item[(2)] Actions of $\mathcal{G}$ and $\mathcal{H}$ on $M$ are both free and proper.

 \item[(3)] The two actions commute with each other.

 \item[(4)] The map $\rho$ is constant on each orbit of the action of $\mathcal{H}$ and an induced map $M/\mathcal{H}\to G_0$ is a~diffeomorphism. Similarly, $\sigma$ is constant on each orbit of the action of $\mathcal{G}$ and an induced map~${\mathcal{G}\backslash M\to H_0}$ is a diffeomorphism.
 \end{enumerate}
 The triplet $(M,\rho,\sigma)$ is called an \textit{equivalence bimodule} from $\mathcal{G}$ to $\mathcal{H}$,
 \[\xymatrix{{G}_1\ar@<-.5ex>[d] \ar@<.5ex>[d]&M\ar[ld]^{\rho}\ar[rd]_{\sigma}&{H}_1\ar@<-.5ex>[d] \ar@<.5ex>[d]\\
 G_0&&H_0.
 }\]
\end{Definition}

Regarding the relationship between the Morita equivalence of cosymplectic groupoids $\mathcal{G}$, $\mathcal{H}$ and that of their symplectic subgroupoids $S_\mathcal{G}$, $S_\mathcal{H}$, we obtain the following.
\begin{Proposition}
 Let $\mathcal{G}=({G}_1\rightrightarrows G_0)$ and $\mathcal{H}=({H}_1\rightrightarrows H_0)$ be Morita equivalent cosymplectic groupoids and $S_\mathcal{G}=(S_{G_1}\rightrightarrows G_0)$, $S_\mathcal{H}=(S_{H_1}\rightrightarrows H_0)$ their symplectic subgroupoids. Let $(M,\rho,\sigma)$ be an equivalence bimodule from $\mathcal{G}$ to $\mathcal{H}$ and assume that there is a symplectic leaf $S$ of $M$ which satisfies the following conditions:
 \begin{itemize}\itemsep=0pt
 \item $\rho|_S\colon S\to G_0$, $ \sigma|_S\colon S\to H_0$ are surjective.
 \item For any $x\in S$ and $g\in G_1$ such that $gx$ is defined, $gx\in S$ implies $g\in S_{G_1}$.
 \item For any $x\in S$ and $h\in H_1$ such that $xh$ is defined, $xh\in S$ implies $h\in S_{H_1}$.
 \end{itemize}
 Then the triplet $(S,\rho|_S,\sigma|_S)$ is an equivalence bimodule from $S_\mathcal{G}$ to $S_\mathcal{H}$, and thus these symplectic groupoids are Morita equivalent,
 \[\xymatrix{S_{G_1}\ar@<-.5ex>[d] \ar@<.5ex>[d]&S\ar[ld]^{\rho|_S}\ar[rd]_{\sigma|_S}&{S_{H_1}}\ar@<-.5ex>[d] \ar@<.5ex>[d]\\
 G_0&&H_0.}\]
\end{Proposition}
\begin{proof}
 First, Proposition~\ref{prop:2} implies that actions of $S_{\mathcal{G}}$ and $S_\mathcal{H}$ preserves the leaf $S$, and these actions are both symplectic.

 Since actions of $\mathcal{G}$, $\mathcal{H}$ are both cosymplectic, ${\rm d}\rho(R)=0$, $ {\rm d}\sigma(R)=0$ holds for the Reeb vector field $R$ of $M$. Hence $\rho|_S$, $\sigma|_S$ are submersions.

 Then $\rho|_S$ is constant along each orbit of the $S_{\mathcal{H}}$-action, and it induces a diffeomorphism $S/S_{\mathcal{H}}\to G_0$ since for $x\in S$, $gx\in S$ implies $g\in S_{G_1}$ and $\rho$ induces a diffeomorphism $M/\mathcal{H}\to G_0$. Similarly, we can see that $\sigma|_S$ induces a diffeomorphism $S_{\mathcal{G}}\backslash S\to H_0$. The other conditions can be easily verified.
\end{proof}

Xu \cite{xu2} studied the notion of Morita equivalence of symplectic groupoids and applied it to investigate Morita equivalence of Poisson manifold \cite{xu1}. In this paper, we defined the notion of a cosymplectic groupoid action and that of Morita equivalence between cosymplectic groupoids. Regarding them, future work includes demonstrating that results parallel to those in the case of symplectic groupoids hold (e.g., whether Morita equivalence between two cosymplectic groupoids implies an equivalence of categories between their module categories).

\subsection*{Acknowledgements}

The author is grateful to R.~Goto for his encouragement. The author also thank N.~Ikeda for useful conversations. He greatly appreciates the suggestions of the anonymous referees which considerably improved the presentation. This work was supported by JSPS KAKENHI Grant Number JP23KJ1487.

\pdfbookmark[1]{References}{ref}
\LastPageEnding

\end{document}